\documentclass{article}
\usepackage{amsmath}
\usepackage{amssymb}
\usepackage{amsthm}
\usepackage{amsxtra}
\usepackage{bbm}
\usepackage{ulsy}
\usepackage{accents}
\usepackage{enumerate}
\usepackage{esint}
\usepackage[ruled]{algorithm2e}
\usepackage{tikz}
\usepackage{multirow,rotating}
\usepackage[ruled]{algorithm2e}

\usepackage{url}
\usepackage[nocompress]{cite}

\usepackage{graphicx}

\usepackage{booktabs}
\usepackage{array}

\numberwithin{equation}{section}

\newtheorem{theorem}{Theorem}

\newtheorem{proposition}{Proposition}

\theoremstyle{remark}
\newtheorem{remark}{Remark}

\providecommand{\abs}[1]{\lvert#1\rvert}
\providecommand{\norm}[1]{\lVert#1\rVert}

\newcommand{\bitem}{\begin{itemize}}
\newcommand{\eitem}{\end{itemize}}

\newcommand{\R}{\mathbb{R}}

\newcommand{\bpm}{\begin{pmatrix}}
\newcommand{\epm}{\end{pmatrix}}
\newcommand{\bsm}{\left(\begin{smallmatrix}}
\newcommand{\esm}{\end{smallmatrix}\right)}

\newcommand{\shrink}{\text{shrink}}

\DeclareMathOperator{\argmin}{arg min}

\DeclareMathOperator{\const}{const}

\DeclareMathOperator{\ddiv}{div}

\DeclareMathOperator{\sign}{sign}

\newcolumntype{M}{>{$\vcenter\bgroup\hbox\bgroup}c<{\egroup\egroup$}}

\usepackage[utf8x]{inputenc}

\usepackage{amsmath}
\usepackage{amssymb}
\usepackage{amsthm}
\usepackage{amsxtra}
\usepackage{bbm}
\usepackage{ulsy}
\usepackage{accents}
\usepackage{enumerate}
\usepackage{esint}
\usepackage[ruled]{algorithm2e}
\usepackage{tikz}
\usepackage{multirow,rotating}

\usepackage{url}
\usepackage{graphicx}
\usepackage[hang]{subfigure}
\usepackage{caption}
\usepackage{float}

\usepackage{booktabs}
\usepackage{array}

\newcolumntype{M}{>{$\vcenter\bgroup\hbox\bgroup}c<{\egroup\egroup$}}

\title{Convex Variational Image Restoration with Histogram Priors}
\author{Paul Swoboda and Christoph Schn\"orr}

\begin{document}

\maketitle

\begin{abstract}
We present a novel variational approach to image restoration (e.g., denoising,
inpainting, labeling) that enables to complement established variational
approaches with a histogram-based prior enforcing closeness of the solution to
some given empirical measure. By minimizing a single objective function, the
approach utilizes simultaneously two quite different sources of information for
restoration: spatial context in terms of some smoothness prior and non-spatial
statistics in terms of the novel prior utilizing the Wasserstein distance
between probability measures. We study the combination of the functional lifting
technique with two different relaxations of the histogram prior and derive a
jointly convex variational approach. Mathematical equivalence of both
relaxations is established and cases where optimality holds are discussed. Additionally, we
present an efficient algorithmic scheme for the numerical treatment of
the presented model. Experiments using the basic total-variation based denoising
approach as a case study demonstrate our novel regularization approach.
\end{abstract}

\section{Introduction} \label{sec:Introduction}
A broad range of powerful variational approaches to low-level image analysis
tasks exist, like image denoising, image inpainting or image labeling~\cite{Mathematical_CV_Handbook, LellmannMultiClass}. It is not
straightforward however to incorporate \emph{directly} into the restoration
process statistical prior knowledge about the image class at hand. Particularly,
handling global statistics as part of a single \emph{convex} variational approach has not been
considered so far.

In the present paper, we introduce a class of variational approaches of the form
\begin{equation} \label{eq:approach-general}
\inf_{u}  F(u) + \lambda R(u) + \nu W(\mu^{u},\mu^{0}),
\end{equation}
where $F(u) + \lambda R(u)$ is any energy functional consisting of a
data fidelity term $F(u)$ and a regularization term $R(u)$,
$W(\mu^{u},\mu^{0})$ denotes the histogram prior in terms of the Wasserstein
distance between the histogram corresponding to the minimizing function $u$ to
be determined and some given histogram $\mu^{0}$ and $\lambda>0$ and $\nu>0$ are parameters weighing the influence of each term. 
We require $R(u)$ to be convex. 
As a case study, we adopt for $R(u) = \text{TV}(u)$, the Total Variation, see~\cite{FunctionsOfBoundedVariation}, 
and $F(u) = \int_{\Omega} f(u(x),x) dx$, where $f$ can also be a nonconvex function. The
basic ROF denoising approach of~\cite{Rudin_Osher_Fatemi_1992} is included in
this approach with $f(u(x),x) = \left(u(x) - u_0(x)\right)^2$, where $u_0$ is the
image to be denoised.

Note that minimizing the second term $R(u)$ in~\eqref{eq:approach-general} entails
spatial regularization whereas the third Wasserstein term utilizes statistical
information that is not spatially indexed in any way. As an illustration, consider the academical
example in figure~\ref{fig:stripes}.
\begin{figure} 
\begin{center}
\includegraphics[width=0.7\textwidth]{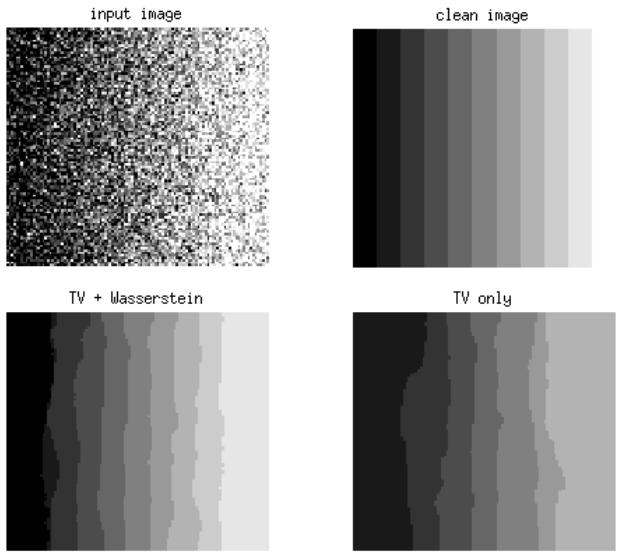}
\end{center}
 \caption{Denoising experiment of a noisy image (upper row, left side) taking
into account statistical prior information through convex optimization (lower
row, left side) infers the correct image structure and outperforms hand-tuned
established variational restoration (lower row, right side). Enforcing
global image statistics to be similar to those of the clean image (upper
row, right side) gives our approach an advantage over methods not taking such
information into account.}
 \label{fig:stripes}
\end{figure}
Knowing the grayvalue distribution of the original image helps us in
regularizing the noisy input image.
 We tackle the corresponding main difficulty in two
different, mathematically plausible ways: by convex relaxations of~\eqref{eq:approach-general} in order to obtain a computationally tractable
approach. Comparing these two relaxations -- one may be tighter than the other
one -- reveals however mathematical equivalence. Preliminary
numerical experiments demonstrate that the relaxation seems to be tight enough
so as to bias effectively variational restoration towards given statistical
prior information.


\section{Prior Work, Contribution} \label{sec:PriorWork}
\subsection{Related Work}
Image regularization by variational methods is a powerful and commonly used tool
for denoising, inpainting, labeling and many other applications. As a case study
in connection with~\eqref{eq:approach-general}, we consider one of the most widely
used approaches for denoising, namely the Rudin, Osher and
Fatemi (ROF) model from~\cite{Rudin_Osher_Fatemi_1992}:
\begin{equation} \label{eq:ROF}
 \min_{u \in \text{BV}(\Omega,[0,1]} \norm{u-u_0}^2 + \lambda \text{TV}(u),
\end{equation}
where $u_0$ is the input image, $\text{TV}$ denotes the Total Variation and
$\text{BV}(\Omega,[0,1])$ is the space of functions of bounded variation with domain $\Omega \subset \R^d$ and values in $[0,1]$. 
The minimization problem~\eqref{eq:ROF} is convex and can be solved to a global optimum 
efficiently by various first-order proximal splitting algorithms even for large problem
sizes, e.g. by Primal-Dual methods~\cite{FastPrimalDualChambollePock} or other proximal minimization algorithms for nonsmooth convex optimization~\cite{ADMM-2010, DouglasRachfordEckstein, raguet-gfb}.

We can also use more general data terms instead of the quadratic term in~\eqref{eq:ROF}. For example in~\cite{GlobalSolutionsOfVariationalModels} it is shown
how the data term can be replaced by a continuous but possibly non-convex function 
$\int_\Omega f(u(x),x) dx$. Still this data function is local and does not take
into account global statistics.

In the case that some prior knowledge is encoded as a histogram, the Wasserstein
distance and the associated Optimal Transport are a suitable choice for
penalizing deviance from prior knowledge. More generally the
Wasserstein distance can be used as a distance on histograms over arbitrary
metricized spaces.

Regarding the Wasserstein distance and the theory of Optimal Transport we refer to the in-depth treatise~\cite{VillaniOptimalTransport}.
Optimal Transport is well-known as Earth Mover's distance in image processing and  computer vision~\cite{EMDRubner} and has been used for content-based image retrieval. 
Further recent applications include~\cite{Chan_histogrambased, peyre-wasserstein-rbac} in connection with segmentation and~\cite{ferradans-mixing-preprint} for texture synthesis.

The authors of~\cite{WassersteinRegularizationImagingRabinPeyre} propose an approach to contrast and color modification.
Given a prior model of how the color or grayvalues are distributed in an image, the authors propose a variational formulation for modifying the given image so that these statistical constraints are met in a spatially regular way.
While their algorithm is fast, high runtime performance is achieved by minimizing a non-convex approximation of their original energy. 
In contrast, we directly minimize a convex relaxation of the original energy, hence we may hope to obtain lower energies and not to get stuck in local minima.

Our variational approach employing the Wasserstein distance as a histogram-based prior through \emph{convex} relaxation appears to be novel.

\subsection{Contribution}
We present
\begin{itemize}
 \item a variational model with a histogram-based prior for image restoration (Section~\ref{sec:ProblemSetting}),
 \item two convex relaxations of the original problem together with discussions of cases where optimality holds (Sections~\ref{sec:SaddlePointFormulation} and~\ref{sec:HoeffdingFrechetBoundsFormulation}),
 \item a proof of equivalence for the two presented relaxations (Section~\ref{sec:RelationshipOfRelaxations}),
 \item an efficient numerical implementation of the proposed variational model (Section~\ref{sec:Numerics}),
 \item experimental validation of the proposed approach (Section~\ref{sec:Experiments}).
\end{itemize}

\section{Problem and Mathematical Background} \label{sec:ProblemSetting}
We introduce the original non-convex model, consider different ways to write
the Wasserstein distance and introduce the functional lifting technique for
rewriting the resulting optimization problem to show well-posedness and to make it amenable for global optimization.
\subsection{Problem Statement}
For an image domain $\Omega \subset \R^2$, e.g. $\Omega = [0,1]^2$ and $u:
\Omega \rightarrow [0,1]$, consider the normalized pushforward $\mathcal{L}_{|\Omega}$ of the Lebesgue measure $\mathcal{L}$ restricted to $\Omega$ by $u$:
\begin{equation} \label{eq:Grey-valueHistogram}
 \mu^{u}(A) = \frac{1}{\mathcal{L}(\Omega)} \left(u_{*} \mathcal{L}\right)(A) = \frac{1}{\mathcal{L}(\Omega)} \mathcal{L}(u^{-1}(A)) \quad \forall A \subset [0,1] \text{ measurable}\,.
\end{equation}
We will use the notation $\abs{B} := \mathcal{L}(B)$ for simplicity. In other words, $\mu^u$ is the grayvalue histogram of the image $u$.
We would like to minimize the energy function
\begin{equation} \label{eq:energy}
\min_{u \in \text{BV}(\Omega, [0,1])} E(u) = \int_{\Omega}
f(u(x),x)  dx + \lambda \text{TV}(u) + \nu W(\mu^{u},\mu^0).
\end{equation}
$\text{TV}(u)$ is the Total Variation
\begin{equation} \label{eq:TV}
 \text{TV}(u) = \sup\left\{\int_\Omega u(x) \ddiv \phi(x) dx\ \colon\ \phi \in
C_c^1(\Omega,\R^2), \norm{\phi}_{\infty} \leq 1\right\}\,,
\end{equation}
where $C_c^1(\Omega,\R^2)$ is the space of continuously differentiable functions with compact support in $\Omega$ and values in $\R^2$,
see~\cite{FunctionsOfBoundedVariation} for more details.
$f: [0,1] \times \Omega \rightarrow \R$ is a continuous fidelity function, and $W$ is the Wasserstein distance
\begin{equation} \label{eq:WassersteinDistance}
 W(\mu,\tilde{\mu}) = \inf_{\pi \in \Pi(\mu,\tilde{\mu})} \int_{[0,1] \times
[0,1]} c(\gamma_1,\gamma_2)\ d\pi(\gamma_1,\gamma_2).
\end{equation}
$c: [0,1] \times [0,1] \rightarrow \R$ is the cost
function for the Wasserstein distance, for example $c(\gamma_1,\gamma_2) =
\abs{\gamma_1-\gamma_2}^p$ with $p\geq 1$. The space of transport plans is 
\begin{equation}
\Pi(\mu,\tilde{\mu})  = \{\pi \in \mathcal{P}([0,1] \times
[0,1]) \ \colon \ 
\begin{array}{l}
\pi( A \times [0,1] ) = \mu(A) \\ 
\pi( [0,1] \times B ) = \tilde{\mu}(B)
\end{array}
\quad \forall A,B \text{ measurable}\},
\end{equation}
where $\mathcal{P}([0,1] \times [0,1])$ is the space of all probability measures
defined on the Borel-$\sigma$-Algebra over $[0,1] \times [0,1]$.
If $c$ is lower semicontinuous and there exist upper semicontinuous functions $a,b\in L_1([0,1])$ such that $c(\gamma_1,\gamma_2) \geq a(\gamma_1) + b(\gamma_2)$, then by Theorem~4.1 in~\cite{VillaniOptimalTransport} there exists a measure
which minimizes~\eqref{eq:WassersteinDistance} and which is called the optimal
transport plan. The optimization problem~\eqref{eq:WassersteinDistance} is linear in both constraints and objective and therefore convex.
Note however that energy~\eqref{eq:energy} is not convex. 

By minimizing~\eqref{eq:energy} we obtain a solution $u$ which remains
faithful to the data by the fidelity term $f$, is spatially coherent by
the Total Variation term and has global grayvalue statistics
similar to $\mu^0$ by the Wasserstein term.

\begin{remark}
In the case of two labels, which means restricting the function $u$ in~\eqref{eq:energy} to have values $u(x) \in \{0,1\}$, our model reduces to foreground/background segmentation and the Wasserstein term can be interpreted as a prior favoring a prespecified size of the foreground area.
\end{remark}

\subsection{The Wasserstein Distance and its Dual}
\label{sec:WassersteinDistance}
We reformulate energy~\eqref{eq:energy}  by introducing another way to obtain the Wasserstein distance.
Assume the cost $c: [0,1] \times [0,1] \rightarrow \R$ is lower semicontinuous such that
\begin{equation}
c(\gamma_1,\gamma_2) \geq a(\gamma_1) + b(\gamma_2) \quad \forall x,y \in [0,1]
\end{equation}
for $a,b \in L^1([0,1])$ upper semicontinuous.

Recall Theorem 5.10 in~\cite{VillaniOptimalTransport}, which
states that the following dual Kantorovich formulation equals the Wasserstein distance:
\begin{equation} \label{eq:DualKantorovich}
W(\mu^u,\mu^0) = 
  \sup_{\substack{(\psi,\psi') \in L_1([0,1])^2\\ \psi(\gamma_1) - \psi'(\gamma_2) \le c(\gamma_1,\gamma_2) }} 
\int_0^1 \psi d \mu^u  - \int_0^1 \psi' d\mu^0.
\end{equation}
Define therefore
\begin{equation} \label{eq:E}
 E(u,\psi,\psi') = \int_{\Omega} f(u(x),x) dx + \lambda \text{TV}(u)
 + \nu \left(\int_0^1 \psi d\mu^u - \int_0^1 \psi' d\mu^0 \right)
\end{equation}
and let 
\begin{equation} 
C = \text{BV}(\Omega,[0,1])
\end{equation}
be the space of functions of bounded variation with domain $\Omega$ and range
$[0,1]$ and
\begin{equation}
D = \left\{\psi, \psi' : [0,1] \rightarrow \R \text{ s.t. } 
\begin{array}{l}
\psi(\gamma_1) - \psi'(\gamma_2) \le c(\gamma_1,\gamma_2)  \ \forall \gamma_1,\gamma_2 \in [0,1] \\ 
\ \psi, \psi' \in L_1([0,1]) 
\end{array} \right\}.
\label{eq:DualKantorovichSet}
\end{equation}
It follows from~\eqref{eq:DualKantorovich} with the above definitions that
\begin{equation} \label{eq:minmax}
\inf_{u \in C} E(u) = \inf_{u \in C} \sup_{(\psi, \psi') \in D} E(u,\psi,\psi')
\end{equation}

\subsection{Functional Lifting} \label{sec:FunctionalLifting}
While the Wasserstein distance~\eqref{eq:WassersteinDistance} is
convex in both of its arguments, see Theorem 4.8 in~\cite{VillaniOptimalTransport}, the energy in~\eqref{eq:energy} is not convex due to the
nonconvex transformation $u \mapsto \mu^u$ in the first argument of the
Wasserstein term and the possible nonconvexity of $f$. 
To overcome the nonconvexity of both the data term and the transformation
in the first argument of the Wasserstein distance we lift the function $u$.
Instead of $u$ we consider a function $\phi$ defined below whose domain is one
dimension larger. This extra dimension represents the range of $u$ and allows
us both to linearize the fidelity term and to convexify the Wasserstein
distance.
This technique, known as functional lifting or the calibration method, was 
introduced in~\cite{CalibrationAlbertiBouchitteDalMaso} and is commonly used in
many optimization problems.

Let
\begin{equation} \label{eq:liftingsetnonconvex}
 C' = \left\{\phi \in \text{BV}\left(\Omega \times \R,\{0,1\}\right) \ \colon \ 
\begin{array}{c}
\phi(\cdot,(-\infty,0]) \equiv 1,\  \phi(\cdot,[1,\infty)) \equiv 0,\\ D_{\gamma} \phi(\cdot,\gamma) \le 0\ \end{array}\right\}\,.
\end{equation}
Every function $u \in C$ corresponds uniquely to a function $\phi \in C'$ via
the relation
\begin{equation} \label{eq:FunctionalLiftingRelation}
-D_{\gamma} \phi = \mathcal{H}^{2}\llcorner \text{graph}(u)\,,
\end{equation}
where $ \mathcal{H}^{2}\llcorner \text{graph}(u)$ is the restriction of the $2$-dimensional Hausdorff measure to the graph of $u$.
Also for such a pair $(u,\phi)$ and for all measurable sets $A \subset [0,1]$ we have the relation
\begin{equation}
\mu^u(A) = \mu^{\phi}(A) = 
\frac{1}{\abs{\Omega}} \int_{\Omega} \abs{D_{\gamma} \phi(x,A)} dx 
= \frac{1}{\abs{\Omega}} \int_{\Omega} -D_{\gamma} \phi(x,A) dx\,.
\end{equation}
Note that in contrast to $u \mapsto \mu^u$, the transformation $\phi \mapsto \mu^{\phi}$ is linear.

Consider the energy
\begin{equation} \label{eq:functionallifting}
E'(\phi,\psi,\psi') = \begin{array} {l}
- \int_{\Omega} \int_{0}^{1} f(\gamma,x) D_{\gamma} \phi(x,\gamma)\,dx 
+ \lambda \int_0^1 \text{TV}( \phi(\cdot,\gamma) ) d\gamma \\
+ \nu \left(\int_0^1 \psi d\mu^{\phi} - \int_0^1 \psi' d\mu^0 \right)\,.
\end{array}
\end{equation}
For a pair $(u,\phi)$ as in~\eqref{eq:FunctionalLiftingRelation} the identity
\begin{equation}
E(u,\psi,\psi') = E'(\phi,\psi,\psi')
\end{equation}
holds true by the coarea formula, see~\cite{FunctionsOfBoundedVariation}. 
Consequently, we have
\begin{equation} \label{eq:LiftedMinimizationProblem}
 \inf_{u \in C} \sup_{(\psi,\psi') \in D} E(u,\psi,\psi') =
 \inf_{\phi \in C'} \sup_{(\psi,\psi') \in D} E'(\phi,\psi,\psi')\,.
\end{equation}
Note that $E'$ is convex in $\phi$ and concave in $(\psi,\psi')$, hence is easier to handle from an optimization point of view.

\begin{theorem}
Let $\Omega \subset \R^2$ be bounded, let $f(x,\gamma)$ be continuous and let the cost $c$ of the Wasserstein distance fulfill the conditions from Section~\ref{sec:WassersteinDistance}.
Then there exists a minimizer $\overline{\phi}$ of $\inf_{\phi \in C'} \sup_{(\psi,\psi') \in D} E'(\phi,\psi,\psi')$.
\end{theorem}
\begin{proof}
We first show that the set $C'$ is compact in the $\text{weak}^*$ topology in BV. By theorem~3.23 in~\cite{FunctionsOfBoundedVariation}, $C'$ is precompact. It then remains to prove that $C'$ is closed in the $\text{weak}^*$-topology. Thus let $(\phi_n)$ in $C'$ converge $\text{weakly}^*$ to $\phi$, which means that $(\phi_n)$ converges strongly in $L^1_{loc}$ and $D_{\gamma} \phi_n$ converges $\text{weakly}^*$.
$D_{\gamma} \phi_n(\cdot,\gamma) \leq 0$ means 
\begin{equation}
\int_{\Omega \times \R} w D_{\gamma} \phi_n \geq 0 \quad  \forall w \in C_c(\Omega \times \R)\,.
\end{equation}
This property is preserved under $\text{weak}^*$-convergence by definition.
$\phi(x,\gamma) \in \{0,1\}$ a.e. as convergence in $L^1$ implies pointwise convergence of some subsequence.
Obviously $\phi_n(\cdot,(\-\infty,0]) \equiv 1$ and $\phi_n(\cdot,[1,\infty)) \equiv 0$ are naturally preserved in the limit.

The first term in the energy~\eqref{eq:functionallifting} is lower semicontinuous by assumption. The TV-term is lower-semicontinuous by Theorem~5.2 in~\cite{FunctionsOfBoundedVariation}.

The Wasserstein term in~\eqref{eq:functionallifting} has the form
$\sup_{\{(\psi,\psi') \in D\}} \int_0^1 \psi\,d \mu^{\phi}  - \int_0^1 \psi'\,d\mu^0$ and can thus be written as
\begin{equation}
\sup_{\{(\psi,\psi') \in D, \psi,\psi' \in C_c([0,1]) \}} -\frac{1}{\abs{\Omega}} \int_0^1 \int_{\Omega} \psi(\gamma) dx\,D_{\gamma} \phi(x,\gamma) - \int_0^1 \psi'(\gamma)\,d\mu^0(\gamma),
\end{equation}
where $C_c([0,1])$ is the space of all continuous functions in $[0,1]$. 
Hence it is a supremum of linear functionals and lsc as well.

As a supremum of positive sums of lsc terms, $\sup_{(\psi,\psi') \in D\cap C_c([0,1])^2} E'(\cdot,\psi,\psi')$ is lsc as well.
A minimizing sequence therefore has a $\text{weakly}^*$-convergent subsequence due to compactness of $C'$. The limit is a minimizer by the lower semicontinuity of the energy.
\end{proof}

As we have shown above, the proposed lifted model is well-posed, which means that the minimizer is attained under mild technical conditions. Then by~\eqref{eq:LiftedMinimizationProblem} also the original energy is well-posed.

\begin{remark}
We have considered a spatially continuous formulation, as
discretizations thereof suffer less from grid bias~\cite{DiscreteAndContinuousModelsForPartitioning, ComparisonDiscreteContinuous} than purely discrete formulations. Thus, proving existence of a solution of the spatially continuous model substantiates our approach from a modelling point of view.
\end{remark}

\begin{remark}
As discussed in Section~\ref{sec:Introduction}, we merely consider total variation based regularization as a case study, but this restriction is not necessary. 
More general regularizers can be used as well as long as they are convex and all the statements still hold, e.g. quadrativ or Huber functions, see~\cite{ConvexFormulationContinuousMultiLabel}.
In the present paper however, we rather focus on the novel prior based on the Wasserstein distance.
\end{remark}

\section{Relaxation as a Convex/Concave Saddle Point Problem}
\label{sec:SaddlePointFormulation}
Optimizing energy~\eqref{eq:energy} is not tractable, as it is a nonconvex
problem. Also solving~\eqref{eq:LiftedMinimizationProblem} is not tractable, as
the set $C'$ is nonconvex. The latter can be overcome by
considering the convex hull of $C'$, which 
leads to a relaxation as a convex/concave saddle point problem of the
minimization problem~\eqref{eq:energy}, which is solvable computationally.
\begin{proposition} \label{prop:SaddlePoint}
 Let 
\begin{equation} \label{eq:liftingset}
 C'' = \{\phi \in BV(\Omega \times \R, [0,1]) \ \colon \ 
\phi(\cdot,(-\infty,0]) \equiv 1,\  \phi(\cdot,[1,\infty)) \equiv 0,\  D_{\gamma} \phi \le 0\}
\end{equation}
Then $E'$ is convex/concave and
\begin{equation} \label{eq:minmaxrelax2}
 \min_{u \in C} E(u) \ge \min_{\phi \in C''} \sup_{(\psi, \psi')
\in D} E'(\phi,\psi,\psi').
\end{equation}
If 
\begin{equation} \label{eq:minmaxequal}
 \min_{u \in C} \max_{(\psi, \psi') \in D} E(u,\psi,\psi') = \max_{(\psi,
\psi') \in D} \min_{u \in C} E(u,\psi,\psi')
\end{equation}
holds, then the above relaxation is exact.
\end{proposition}
\begin{proof}
 Note that $C''$ is a convex set, in particular it is the convex hull of $C'$.
$E'$ is also convex in $\phi$,
therefore the right side of~\eqref{eq:minmaxrelax2} is a convex/concave saddle
point problem.
For fixed $(\psi, \psi')$ we have the following equality:
\begin{equation} \label{eq:liftingexactness}
 \min_{u \in C} E(u,\psi,\psi') = \min_{\phi \in C''} E'(\phi,\psi,\psi'),
\end{equation}
which is proved in~\cite{GlobalSolutionsOfVariationalModels}.
Thus
\begin{equation}
\begin{array}{rcl}
 \min_{u \in C} E(u) &=& \min_{u \in C} \sup_{(\psi, \psi') \in D}
E(u,\psi,\psi') \\
 & \overset{(*)}{\ge} & \sup_{(\psi, \psi') \in D} \min_{u \in C} 
E(u,\psi,\psi') \\
 & \overset{(**)}{=} & \sup_{(\psi, \psi') \in D} \min_{\phi \in C''} 
E'(\phi,\psi,\psi'),
\end{array}
\end{equation}
where $(*)$  is always fulfilled for minimax problems and $(**)$ is
a consequence of~\eqref{eq:liftingexactness}. This proves
\eqref{eq:minmaxrelax2}. If~\eqref{eq:minmaxequal} holds, then $(*)$ above is
actually an equality and the relaxation is exact.
\end{proof}


\section{Relaxation with Hoeffding-Fr\'{e}chet Bounds}
\label{sec:HoeffdingFrechetBoundsFormulation}
A second relaxation can be constructed by using the primal formulation~\eqref{eq:WassersteinDistance} of the  Wasserstein distance and
enforcing the marginals of the distribution function of the transport plan to
be $\mu^\phi$ and $\mu^0$ by the Hoeffding-Fr\'{e}chet bounds: 
\begin{theorem}[{\cite[Thm.~3.1.1]{RachevRueschendorfMassTransportationI}}]\
 Let $F_1, F_2$ be two real distribution functions (d.f.s) and $F$ a d.f. on
$\R^2$. Then $F$ has marginals $F_1, F_2$, if and only if
\begin{equation} \label{eq:HoeffdingFrechetBounds}
 (F_1(\gamma_1) + F_2(\gamma_2) - 1)_{+} \le F(\gamma_1,\gamma_2) \le \min\{F_1(\gamma_1), F_2(\gamma_2)\}
\end{equation}
\end{theorem}

By~\eqref{eq:WassersteinDistance} the Wasserstein Distance with marginal d.f.s
$F_1, F_2$ can be computed by solving the optimal transport problem and we arrive
at the formulation
\begin{equation} \label{eq:WassersteinOptimalTransport}
 W(dF_1, dF_2) = \min_{F} \int_{\R^2} c(dF_1,dF_2)\ dF, \quad \text{ s.t. } F \text{
respects the conditions }\eqref{eq:HoeffdingFrechetBounds}
\end{equation}
where $dF_i$ shall denote the measure associated to the d.f. $F_i$, $i=1,2$.

Using again the functional lifting technique of~\cite{GlobalSolutionsOfVariationalModels},
the Hoeffding-Fr\'{e}chet bounds and the representation of the Wasserstein
distance~\eqref{eq:WassersteinOptimalTransport}, we arrive at the following
relaxation, 
where we replace the distribution functions $F_1$ by the distribution function of $\mu^{\phi}$, which is
$\int_{\Omega} -D_{\gamma} \phi(x,[0,\gamma]) dx$.
\begin{equation} \label{eq:relaxation3}
\begin{array}{rl}
 \min_{\phi, F} &
 \int_{\Omega} \int_0^1 - f(\gamma,x) D_{\gamma}\phi(x,\gamma) dx 
+ \lambda \int_0^1 \text{TV}( \phi(\cdot,\gamma) d\gamma 
+ \nu \int_{\R^2} c\ dF, \\
s.t. & F_{\phi}(\gamma) = \frac{1}{\abs{\Omega}} \int_{\Omega} -D_{\gamma} \phi(x,[0,\gamma]) dx,\\
 & F_{\mu^0}(\gamma) = \mu^0([0,\gamma]), \\
& F_{\phi}(x_1) + F_{\mu^0}(x_2) - 1 \le F(x_1,x_2) \le \min\{F_{\phi}(x_1),
F_{\mu^0}(x_2)\} \\
& \phi \in C''
\end{array}
\end{equation}
The minimization problem~\eqref{eq:relaxation3} is a relaxation of~\eqref{eq:energy}. Just set 
$$\phi(x,\gamma) = \left\{ 
\begin{array}{rl} 1,& u(x) < \gamma \\ 
                  0,& u(x) \ge \gamma \\\end{array}
\right.$$
 and let $F$ be the d.f. of the optimal transport measure with marginals
$\mu^u$ and $\mu^0$.

\begin{remark}
It is interesting to know, when relaxation~\eqref{eq:relaxation3} is exact. By
the coarea formula~\cite{Ziemer-89} we know that 
\begin{equation}
\begin{array}{rl}
&  \int_{\Omega} \int_0^1  - f(\gamma,x) D_{\gamma} \phi(x,\gamma) dx   + \lambda \int_0^1 \text{TV}( \phi(\cdot,\gamma) ) d\gamma\\
= &
\int_0^1 \int_{\Omega} f(u_{\alpha}(x),x) dx d\alpha + \lambda \int_0^1 \text{TV}(u_{\alpha}) d\alpha\,,
\end{array}
\end{equation}
where $u_{\alpha}$ corresponds to the thresholded function $\phi_{\alpha} = \mathbbmss{1}_{\{\phi >
\alpha\}} \in C'$ via relation~\eqref{eq:FunctionalLiftingRelation}. However such a
formula does not generally hold for the optimal transport: 
Let  $\phi_{\alpha} = \mathbbmss{1}_{\{\phi > \alpha\}}$
and let $F_{\alpha}$ be the d.f. of the optimal coupling with marginal d.f.s
$F_{\phi_{\alpha}}$ and $F_{\mu^0}$. Then
\begin{equation}
 F = \int_0^1 F_{\alpha}\ d\alpha
\end{equation}
has marginal d.f.s $\int_0^1 F_{\phi_{\alpha}} d\alpha$ and $F_{\mu^0}$, but
it may not be optimal.

A simple example for inexactness can be constructed as follows: Let the data
term be $f \equiv 0$ and let $\mu^0 = \frac{1}{2}(\delta_0 + \delta_1)$ and let the cost for
the Wasserstein distance be $c(\gamma_1,\gamma_2) = \lambda \abs{\gamma_1 - \gamma_2}$. Every constant
function with $u(x) = \const \in [0,1]$ will be a minimizer if $\lambda$ is
small and $\nu$ is big enough. The objective value will be $\frac{\lambda}{2}$. But relaxation~\eqref{eq:relaxation3} is inexact in this situation: Choose $\phi(x,\gamma) = \frac{1}{2} \quad \forall \gamma \in (0,1)$ and the relaxed
objective value will be $0$.
\end{remark}

\begin{remark}
The above remark was concerned with an example, where a convex combination of optimal solutions to the non-relaxed problem is a unique solution of the relaxed problem with lower objective value.

By contrast, in Section~\ref{sec:Experiments} two different academical examples are shown, which illustrate the behaviour of our relaxation~\eqref{eq:relaxation3} in situations when the non-relaxed solution is unique, see Figures~\ref{fig:ExactExample} and~\ref{fig:InexactExample}. 
Then exactness may hold or not, depending on the geometry of level sets of solutions. 
No easy characterization seems to be available for the exactness of model~\eqref{eq:relaxation3}.
\end{remark}

\section{Relationship between the two Relaxations}
\label{sec:RelationshipOfRelaxations}

Both relaxations from Sections~\ref{sec:SaddlePointFormulation} and~\ref{sec:HoeffdingFrechetBoundsFormulation} seem to be plausible but seemingly
different relaxations. Their different nature reveals itself also in the
conditions for which exactness was established. While the condition in
Proposition~\ref{prop:SaddlePoint} depends on the gap introduced by interchanging the
minimum and maximum operation, relaxation~\eqref{eq:relaxation3} is exact if a
coarea formula holds for the optimal solution. It turns out, however, that both 
equations are equivalent, hence both optimality conditions derived
in Sections~\ref{sec:SaddlePointFormulation} and~\ref{sec:HoeffdingFrechetBoundsFormulation} can be used to ensure exactness
of a solution to either one of the relaxed minimization problems.
\begin{theorem} \label{thm:relaxationEquivalence}
 The optimal values of the two relaxations
\eqref{eq:minmaxrelax2} and~\eqref{eq:relaxation3} are equal. 
\end{theorem}
\begin{proof}
 It is a well known fact that
\begin{equation} \label{eq:GeneralSaddlePointProblem}
 \min_{x \in X} \max_{y \in Y} \langle Kx,y \rangle + G(x) - H^*(y)
\end{equation}
and
\begin{equation} \label{eq:GeneralPrimalProblem}
 \min_{x \in X} H(Kx) + G(x)
\end{equation}
are equivalent, where $G: X \rightarrow [0,\infty ]$ and 
$H^*: Y \rightarrow [0,\infty ]$ are proper, convex, lsc
functions, $H^*$ is the convex conjugate of  $H$ and $X$ and $Y$ are two real vector spaces, see~\cite{RockafellarConvexAnalysis} for details.

To apply the above result choose  
\begin{equation}
 G(\phi) = \int_0^1 \int_{\Omega} -D_{\gamma} \phi(x,\gamma) \cdot f(\gamma,x) dx
+ \lambda \int_0^1 \text{TV}( \phi(\cdot,\gamma) ) d\gamma + \chi_{C''}(\phi),
\end{equation}
\begin{equation}
 H^*(\psi,\psi') = \nu \int_0^1 \psi' d\mu^0 +\chi_{D}(\psi,\psi')
\end{equation}
and
\begin{equation}
\begin{array} {c}
K : BV(\Omega \times \R,[0,1]) \rightarrow \mathcal{M}([0,1])^2, \\
 K(\phi) = (\nu \mu^{\phi},0)
\end{array}
\end{equation}
where $C''$ is defined by~\eqref{eq:liftingset}, $D$ by~\eqref{eq:DualKantorovichSet}, $\chi_{C''}(\cdot)$ and $\chi_{D}(\cdot)$
denote the indicator functions of the sets~$C''$ and~$D$ respectively and
$\mathcal{M}([0,1])$ denotes the space of measures on~$[0,1]$.
\eqref{eq:GeneralSaddlePointProblem} corresponds with the above choices to the
saddle point relaxation~\eqref{eq:minmaxrelax2}.

Recall that $H = (H^*)^*$ if $H$ is convex and lsc, i.e. $H$ is the Legendre-Fenchel bidual of itself,
see~\cite{RockafellarConvexAnalysis}. Hence, for positive measures $\mu,\tilde{\mu}$,
the following holds true:
\begin{equation}
\begin{array}{rl}
 H(\mu,\tilde{\mu}) =&  \sup_{\psi,\psi'} \{\int_0^1 \psi d\mu - \int_0^1 \psi' d\tilde{\mu}
 -H^*(\psi,\psi') \} \\
 =&
 \sup_{(\psi,\psi') \in D} \{\int_0^1 \psi d\mu - \int_0^1 \psi' d\tilde{\mu} - \nu \int_0^1
\psi' d\mu^0 \} \\
=& 
\sigma_{D}( \mu,\tilde{\mu}+ \nu \mu^0) \\
\overset{(*)}{=}&
W(\mu,\tilde{\mu} + \nu \mu^0)
\end{array}
\end{equation} 
where $\sigma_{A}(x) = \sup_{a\in A} \langle a,x \rangle$ is the support
function  of the set $A$ and $\nu$ is the weight for the Wasserstein term in~\eqref{eq:E}.
To prove $(*)$, we invoke Theorem 5.10 in~\cite{VillaniOptimalTransport}, which
states that
\begin{equation}
 \sigma_D( \mu, \tilde{\mu}) = \sup_{(\psi,\psi') \in D} \int_0^1 \psi d\mu - \int_0^1
\psi' \tilde{\mu} = 
\min_{\pi \in \Pi(\mu, \tilde{\mu})} \int_{[0,1]^2} c(\gamma_1,\gamma_2) d\pi(\gamma_1,\gamma_2) = W(\mu,\tilde{\mu}),
\end{equation}
and we have infinity for measures which do not have the same mass.

Thus, the energy in~\eqref{eq:GeneralPrimalProblem} can be written as
\begin{equation}
G(\phi) + H(\nu \mu^{\phi},0)
=  G(\phi) + W(\nu \mu^{\phi}, \nu \mu^0)
=  G(\phi) + \nu W( \mu^{\phi},  \mu^0)\,.
\end{equation}
This energy is the same as in relaxation~\eqref{eq:relaxation3}, which concludes the proof.
\end{proof}

\section{Optimization} \label{sec:Numerics}
We present five experiments and the numerical method used to compute them.
\subsection{Implementation}
First, we discretize the image domain $\Omega$ to be $\{1,\ldots,n_1\} \times \{1,\ldots, n_2 \}$ and use forward differences as the gradient operator.
Second, we discretize the infinite dimensional set $C''$ and denote it by
\begin{equation}
 C''_d = \left\{ \phi:\Omega \times \left\{0,\frac{1}{k},\ldots,\frac{k-1}{k}, 1\right\} \rightarrow [0,1] \ \colon\ 
\begin{array}{l}
\phi(\cdot,1) = 0,\, \phi(\cdot,0) = 1,\\ \phi\left(\cdot,\frac{l}{k}\right) \leq \phi\left(\cdot,\frac{l-1}{k}\right)
\end{array}
\right\}\,.
\end{equation}
Hence we consider only finitely many grayvalues an image can take.
The dual Kantorovich set for the discretised problem is then
\begin{equation}
D_d = \left\{\psi, \psi' : \left\{0,\frac{1}{k},\ldots,\frac{k-1}{k}, 1\right\} \rightarrow \R \ \colon\ 
\begin{array}{l}
\psi(\gamma_1) - \psi'(\gamma_2) \le c(\gamma_1,\gamma_2)  \ \forall \gamma_1,\gamma_2 \end{array} \right\}.
\label{eq:DualKantorovichSetDiscretized}
\end{equation}

After computing a minimizer $\phi^*$ of the discretized energy, we threshold it at the value $0.5$ to obtain $\phi^{\ast} = \mathbbmss{1}_{\{\phi^* > 0.5\}}$ and then calculate $u^{\star}$ by the discrete analogue of relation~\eqref{eq:FunctionalLiftingRelation}.

For computing a minimizer of the discretized optimization problem 
\begin{equation}
\min_{\phi \in C''_d} \max_{(\psi,\psi') \in D_d} E'_d(\phi,\psi,\psi')
\end{equation}
it is expedient to use first order algorithms like~\cite{FastPrimalDualChambollePock, ADMM-2010, DouglasRachfordEckstein, raguet-gfb} as the dimensionality of the problem is high. 
To use such algorithms it is necessary to split the function $\max_{(\psi,\psi') \in D_d} E'_d(\phi,\psi,\psi')$ into a sum of terms,  whose proximity operators can be computed efficiently.
Hence consider the following equivalent minimization problem:
\begin{equation}
\min_{\phi \in C''_d,  g \in (\R^{n1 \times n2 \times k \times 2})} 
\langle \tilde{f}, \phi \rangle + \norm{g}_1 + \chi_{\{(u,v)\colon\nabla u = v\}}(\phi,g) + \chi_{C''}(\phi) + W(\mu^{\phi}, \mu^0)\,,
\end{equation}
where $\tilde{f}$ comes from the local cost factor in~\eqref{eq:functionallifting} and $\chi_A$ is the indicator function of a set $A$.
The proximity operator of a function $G$ is defined as
\begin{equation} \label{eq:SplitEnergy}
\text{prox}_{G}(x) = \argmin_{x'} \frac{1}{2} \norm{x-x'}^2 +
G(x')\,.
\end{equation}

The proximity operator of the term $\norm{g}_1$ is the soft-thresholding operator. 

$\text{prox}_{\chi_{\{(u,v)\colon\nabla u = v\}}}(\phi,g)$  can be efficiently computed with Fourier transforms, see for example~\cite{raguet-gfb}.

$\text{prox}_{\chi_{C''}}$ is the projection onto the set  of non-increasing
sequences $C''$. To compute this projection, we
employ the algorithm proposed in~\cite{Chambolle2012}, Appendix D. It is
trivially parallelisable and converges in a finite number of iterations.

Finally, the proximity operator for the Wasserstein distance can be computed
efficiently in some special cases, as discussed in the next Section~\ref{sec:WassersteinProximation}.

We can either use~\cite{raguet-gfb} to minimize~\eqref{eq:SplitEnergy} directly, 
which is equivalent to using the Douglas-Rachford method~\cite{DouglasRachfordEckstein} on a suitably defined product space and absorbing the linear term in the functions in~\eqref{eq:SplitEnergy}.

\subsection{Wasserstein Proximation for $c(\gamma_1,\gamma_2) = \abs{\gamma_1-\gamma_2}$ by soft-thresholding}\label{sec:WassersteinProximation}
In general, computing the proximity operator for the Wasserstein distance can be
expensive and requires solving a quadratic program. However, for the real line
and convex costs, we can compute the proximity operator more efficiently.
One algorithm for the cost function $c(\gamma_1,\gamma_2) = \abs{\gamma_1-\gamma_2}$ is presented below.

The proximation for the weighted Wasserstein distance is
\begin{equation} \label{eq:WassersteinProximation}
\argmin_{\phi} \frac{1}{2} \norm{\phi^0 - \phi}^2_2 + \lambda W(\mu^{\phi},
\mu^0).
\end{equation}
For the special case we consider here, there is a simple expression for the
Wasserstein distance:
\begin{proposition}[{\cite{RachevRueschendorfMassTransportationI}}] \label{prop:WassersteinDistanceRealLine}
For two measures $\mu^1, \mu^2$ on the real line and $c(\gamma_1,\gamma_2) = \abs{\gamma_1-\gamma_2}$, the
Wasserstein distance is
\begin{equation}
 W(\mu^1,\mu^2) = \int_{\R} \abs{ F_{\mu^1}(\gamma) - F_{\mu^2}(\gamma)}\ d\gamma
\end{equation}
\end{proposition}
%
%
%
%
%
%
Due to $D_{\gamma} \phi(x,\gamma) \leq 0$ and $\phi(x,0) = 1$, we can also write
$F_{\mu^\phi}(\gamma)$ as
\begin{equation} \label{eq:CDFPhi}
 F_{\mu^\phi}(\gamma) =  \frac{1}{\abs{\Omega}} \int_{\Omega}  -D_{\gamma} \phi(x,[0,\gamma])\ dx =
\frac{1}{\abs{\Omega}} \int_{\Omega} 1 - \phi(x,\gamma) dx\,.
\end{equation}

Next we show how to solve in closed form the proximity operator for the
Wasserstein distance in the present case.
\begin{proposition}
\label{prop:WassersteinProximityOperator}
Given $\phi^0$, $\lambda >0$, the optimal $\tilde{\phi}$ for the proximity
operator 
 \begin{equation}
  \tilde{\phi} = \argmin_{\phi}  \frac{1}{2} \norm{\phi - \phi^0}_2^2 + \lambda
W(F_{\mu^\phi}, \mu^0)
 \end{equation}
is determined by
\begin{equation}
 \tilde{\phi}(x,\gamma) = \phi(x,\gamma) + c_{\gamma}\,,
\end{equation}
where 
\begin{equation}
c_{\gamma} = \shrink\left(-\frac{1}{\abs{\Omega}} \int_{\Omega}
\phi^0(x,\gamma) dx -
F_{\mu^0}(\gamma) + 1 , \frac{\lambda}{\abs{\Omega}} \right) + 
\frac{1}{\abs{\Omega}} \int_{\Omega}
\phi^0(x,\gamma) dx + F_{\mu^0}(\gamma) - 1
\end{equation}
and shrink denotes the soft-thresholding operator defined componentwise by
\begin{equation}
 \shrink(a,\lambda)^i = (\abs{a^i} - \lambda)_+ \cdot \sign(a^i)
\end{equation}
for $a \in \R^n$, $\lambda > 0$.
\end{proposition}
\begin{proof}

By proposition~\ref{prop:WassersteinDistanceRealLine} and the
characterisation of $F_{\mu^\phi}$ in~\eqref{eq:CDFPhi}, proximation~\eqref{eq:WassersteinProximation} reads
\begin{equation} \label{eq:WassersteinProximationRealLine}
 \argmin_{\phi} \frac{1}{2}\norm{\phi^0 - \phi}_2^2 + \lambda \int_{\R} \abs{1 -
\left(\frac{1}{\abs{\Omega}} \int_{\Omega} \phi(x,\gamma)dx\right) -
F_{\mu^0}(\gamma) }
d\gamma.
\end{equation}
Note that~\eqref{eq:WassersteinProximationRealLine} is an independent
optimization problem for each $\gamma$. Thus, for each $\gamma$ we have to solve
the problem
\begin{equation} \label{eq:WassersteinProximationRealLineDecoupled}
 \argmin_{\phi(\cdot,\gamma)} \frac{1}{2}\norm{\phi^0(\cdot,\gamma) -
\phi(\cdot,\gamma) }_2^2 + \lambda
\abs{1 - \left(\frac{1}{\abs{\Omega}} \int_{\Omega} \phi(x,\gamma) dx\right) -
F_{\mu^0}(\gamma)} .
\end{equation}
It can be easily verified that the solution to problem~\eqref{eq:WassersteinProximationRealLineDecoupled} is
$\phi^0(\cdot,\gamma) + c_\gamma$, where $c_\gamma \in \R$
and 
\begin{equation}
 c_\gamma \in \argmin_{c \in \R} \frac{1}{2} \abs{\Omega} c^2 + \lambda \abs{
\frac{1}{\abs{\Omega}} \int_{\Omega} \phi^0(x,\gamma) dx + c + F_{\mu^0}(\gamma)
- 1}
\end{equation}
and hence 
\begin{equation}
c_\gamma = \shrink\left(-\frac{1}{\abs{\Omega}} \int_{\Omega} \phi^0(x,\gamma) -
F_{\mu^0}(\gamma) + 1, \frac{\lambda}{\abs{\Omega}}\right) +
\frac{1}{\abs{\Omega}} \int_{\Omega}
\phi^0(x,\gamma)dx + F_{\mu^0}(\gamma) - 1.
\end{equation}
\end{proof}

For the discretized problem one just needs to replace integration with summation to obtain the proximation operator.
Concluding, the cost for the Wasserstein proximal step is linear in the size of the input data. 

\begin{remark}
 We have seen in Proposition~\ref{prop:WassersteinDistanceRealLine} that 
$W_1(\mu^{\phi} , \mu^0 ) = \norm{H\phi − (1 − F\mu^0 }_ 1$, 
where $H$ is an operator corresponding to a tight frame, i.e. $H H^* = \abs{\Omega}^{−1}$, 
hence it is also possible to derive Proposition~\ref{prop:WassersteinProximityOperator} by known rules for proximity operators involving composition with tight frames and translation.
\end{remark}

\section{Numerical Experiments} \label{sec:Experiments}

We want to show experimentally
\begin{enumerate}
 \item that computational results conform to the mathematical model, 
\item that the convex relaxation is reasonable.
\end{enumerate}
Note that we do not claim to achieve the best denoising or inpainting results
and we do not wish to compete with other state-of-the-art methods here.
We point out again that the Wasserstein distance can be used together
with other variational approaches to enhance their performance, e.g. with
nonlocal total variation based denoising, see~\cite{NonlocalOperatorsGuyGilboa}.

\begin{remark}
As detailed in Section~\ref{sec:FunctionalLifting}, we lift our functional, so that it has one additional dimension, 
thereby increasing memory requirements and runtime of our algorithm. Non-convex approaches like~\cite{WassersteinRegularizationImagingRabinPeyre} do not have such computational requirements.
Still, the viability of the lifting approach we use was demonstrated in~\cite{ConvexFormulationContinuousMultiLabel} for our variational model without the Wasserstein term. Also all additional operations our algorithm requires can be done very fast on recent graphic cards, hence the computational burden is tractable. 
\end{remark}

We have generally chosen the parameters $\lambda,\nu$ by hand to obtain reasonable results, if not stated differently.

In the \textbf{first experiment} we compare total variation denoising and total
variation denoising with the Wasserstein term for incorporating prior
knowledge. The data term is $f(s,x) = (u_0(x) - s)^2$, where
$u_0$ is the noisy image in figure~\ref{fig:stripes}. 
The cost for the Wasserstein distance is $c(\gamma_1,\gamma_2) = \nu \abs{\gamma_1 - \gamma_2}, \ \nu > 0$. To
ensure a fair comparison, the parameter $\lambda$ for total variation
regularization \emph{without} the Wasserstein term was \emph{hand-tuned in all
experiments} to obtain best results. The histogram was chosen to match the
noiseless image. See Figure~\ref{fig:stripes} for the results. 

Note the trade-off one always has to make for pure total variation denoising: If
one sets the regularization parameter $\lambda$ high, the resulting grayvalue
histogram of the recovered image will be similar to the noisy input image and
generally far away from the histogram of ground truth. By choosing lower data
fidelity and higher regularization strength we may obtain a valid geometry of
the image, however then the grayvalue histogram tends to be peaked at one mode,
as total variation penalizes scattered histograms and tries to draw the modes
closer to each other, again letting the recovered grayvalue histogram being
different from the desired one.
By contrast, the Wasserstein prior in~\eqref{eq:E} guarantees a correct grayvalue histogram also with strong spatial regularization.

The \textbf{second} set of experiments illustrates where exactness of our relaxation may hold or fail, depending on the geometry of the level sets of solutions, see Figures~\ref{fig:ExactExample} and~\ref{fig:InexactExample}.
The gray area is to be inpainted with a Wasserstein prior favoring the gray area to be partly black and partly white. 
Note that both settings illustrate cases, when the global Wasserstein term is indispensable, as otherwise there would be completely no control over how much of the area to be inpainted ends up being white or black.
While our relaxation is not exact for the experiment in Figure~\ref{fig:InexactExample}, thresholding at $0.5$ still gives a reasonable result. 

\begin{figure}
\begin{minipage}[t]{0.49\textwidth}
\subfigure[hang][The gray area is to be inpainted with partly black and white, with slightly more white.]{
\includegraphics[width=0.45\textwidth]{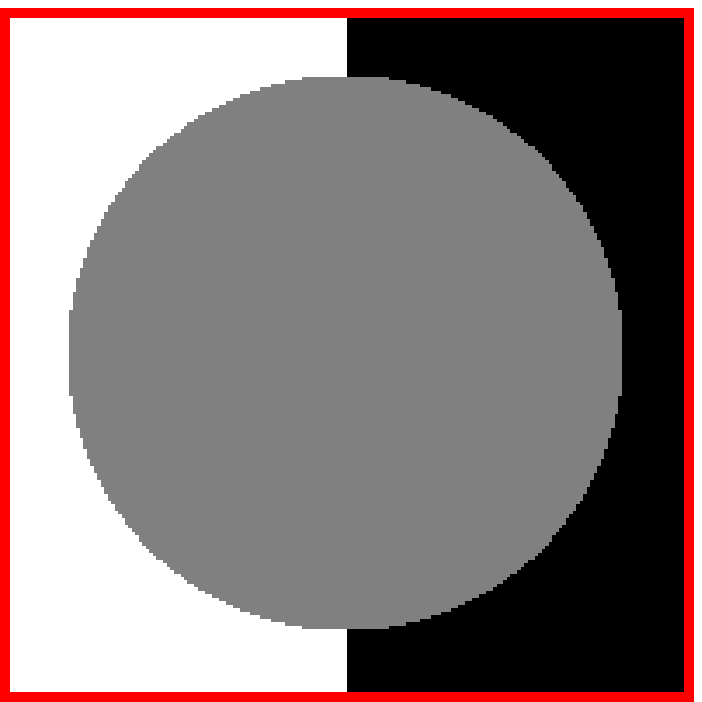}
}
\subfigure[hang][The circle in the middle has been inpainted with slightly more white as demanded by the Wasserstein term.]{
 \includegraphics[width=0.45\textwidth]{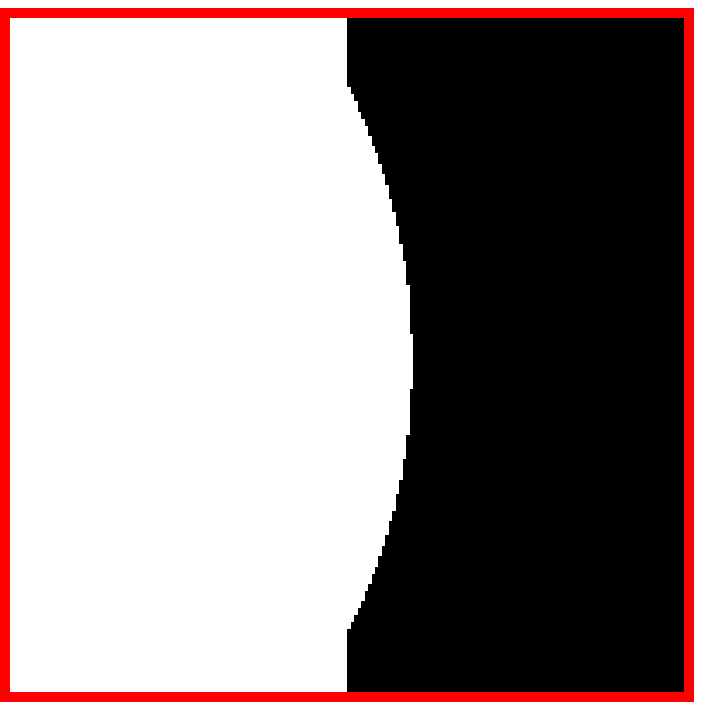}
}
\end{minipage}
\hspace{0.01\textwidth}
\begin{minipage}[t]{0.49\textwidth}
\subfigure[hang][The gray area is the area to be inpainted with a given Wasserstein prior favoring the gray area to be half black and half white.]{
\includegraphics[width=0.45\textwidth]{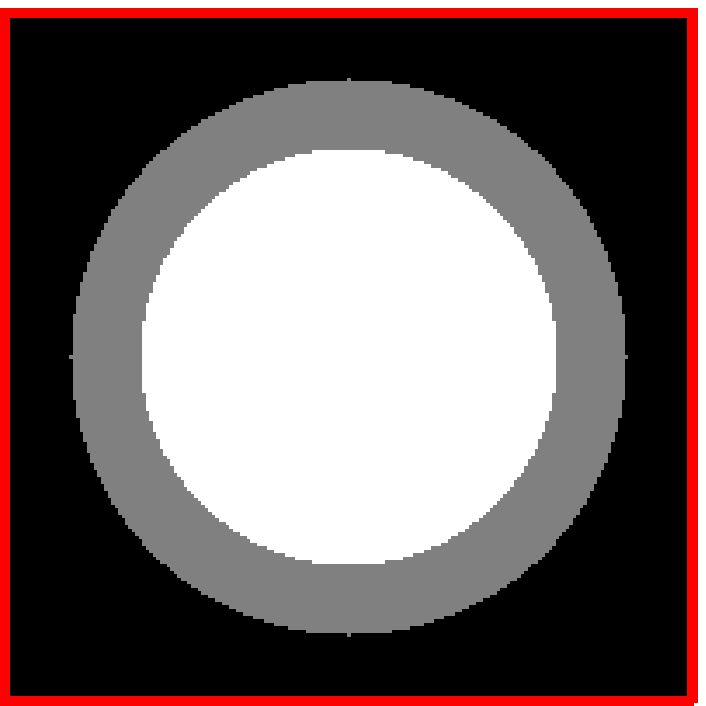}
}
\subfigure[hang][Inpainting result: we obtain a non-integral solution visualized by gray color.]{
 \includegraphics[width=0.45\textwidth]{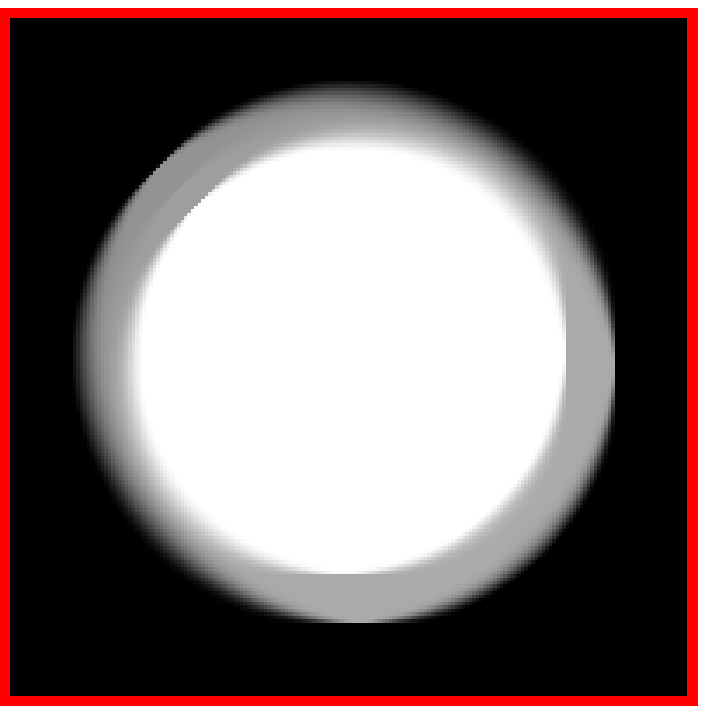}
}
\end{minipage}
\begin{minipage}[t]{0.49\textwidth}
 \caption{Example illustrating tightness of our relaxation~\eqref{eq:relaxation3}.}
\label{fig:ExactExample}
\end{minipage}
\hspace{0.01\textwidth}
\begin{minipage}[t]{0.49\textwidth}
\caption{Example illustrating failure of tightness of our relaxation~\eqref{eq:relaxation3}.}
\label{fig:InexactExample}
\end{minipage}
\end{figure}

The \textbf{third} experiment is a more serious denoising experiment. Notice
that again pure total variation denoising does not preserve the white and black
areas well, but makes them gray, while the approach with the Wasserstein
distance preserves the contrast better, see Figure~\ref{fig:tiger_denoising}.

\begin{figure}
\subfigure[Tiger denoising experiment with the original image on the left, the
image denoised with the Wasserstein term in the middle and the standard
ROF-model on the right.]{
\includegraphics[width=\textwidth]{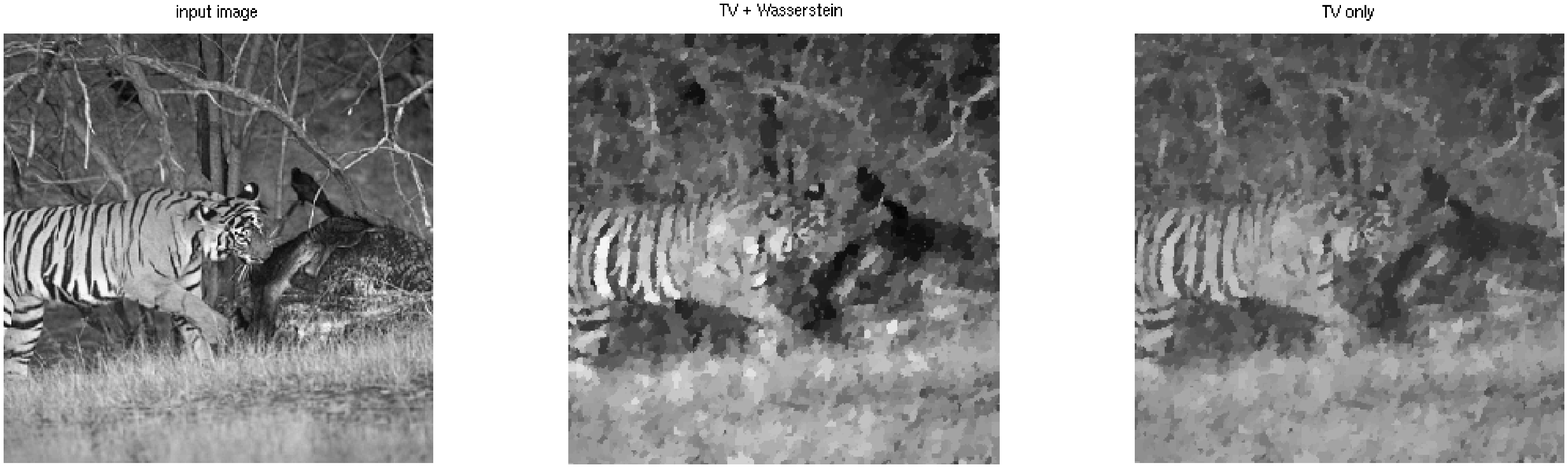}
}
\subfigure[Detailed view of the tiger denoising experiment revealing that
contrast is better pre\-served when the Wasser\-stein term is used.]{
 \includegraphics[width=\textwidth]{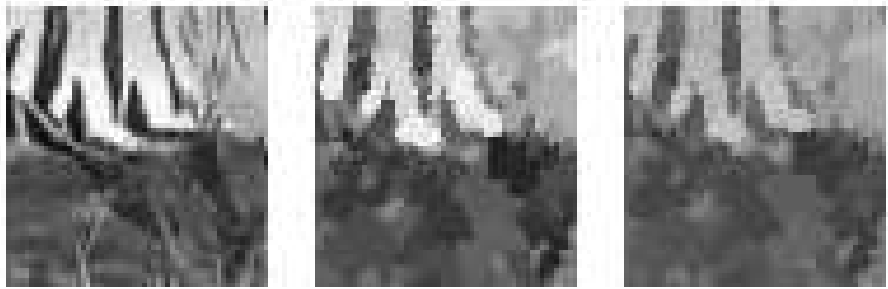}
}
\caption{Tiger denoising experiment}
 \label{fig:tiger_denoising}
\end{figure}

In the \textbf{fourth experiment} we compare image inpainting with a total
variation regularization term without prior knowledge and with prior knowledge,
see Figure~\ref{fig:triple_point} for the results. The
region where the data term is zero is enclosed in the blue rectangle.
Outside the blue rectangle we employ a quadratic data term as in the first
experiment. Total variation inpainting without the Wasserstein term does not produce
the results we expected, as the total variation term is smallest, when the gray
color fills most of the area enclosed by the blue rectangle. Heuristically, this
is so because the total variation term weighs the boundary length multiplied by
the difference between the gray value intensities, and a medium intensity
minimizes this cost. Thus the TV-term tends to avoid
interfaces, where high and low intensities meet, preferring smaller intensity
changes, which can be achieved by interfaces with gray color on one side.
Note that also the regularized image with the Wasserstein term lacks
symmetry. This is also due to the behaviour of the TV-term described above.

\begin{figure} 
 \includegraphics[width=\textwidth]{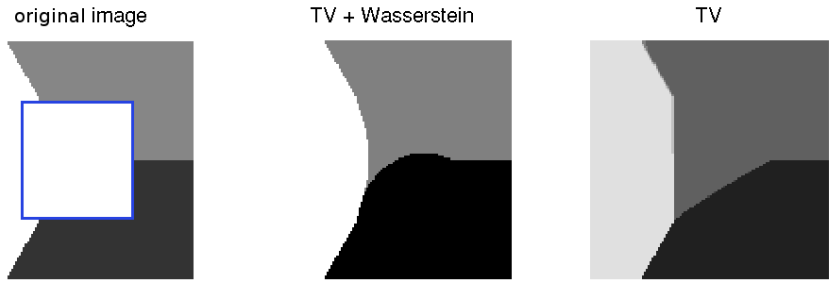}
 \caption{Inpainting experiment with the original image and the inpainting area
enclosed in a blue rectangle on the left, the inpainting result with the
Wasserstein term in the middle and the result where only the TV-regularizer is
used on the right. By enforcing the three regions to have the same size with the
Wasserstein term, we obtain a better result than with the Total Variation term
alone.}
 \label{fig:triple_point}
\end{figure}

In the \textbf{fifth} experiment we consider inpainting again. Yevgeni Khaldei,
the photographer of the iconic picture shown on the left of Figure~\ref{fig:Reichstag_big} had to remove the second watch.
Trying to inpaint the wrist with a TV-regularizer and a Wasserstein term results
in the middle picture, while only using a TV-regularizer results in the right
picture. Clearly using the Wasserstein term helps.

\begin{figure} 
 \includegraphics[width=0.35\textwidth]
{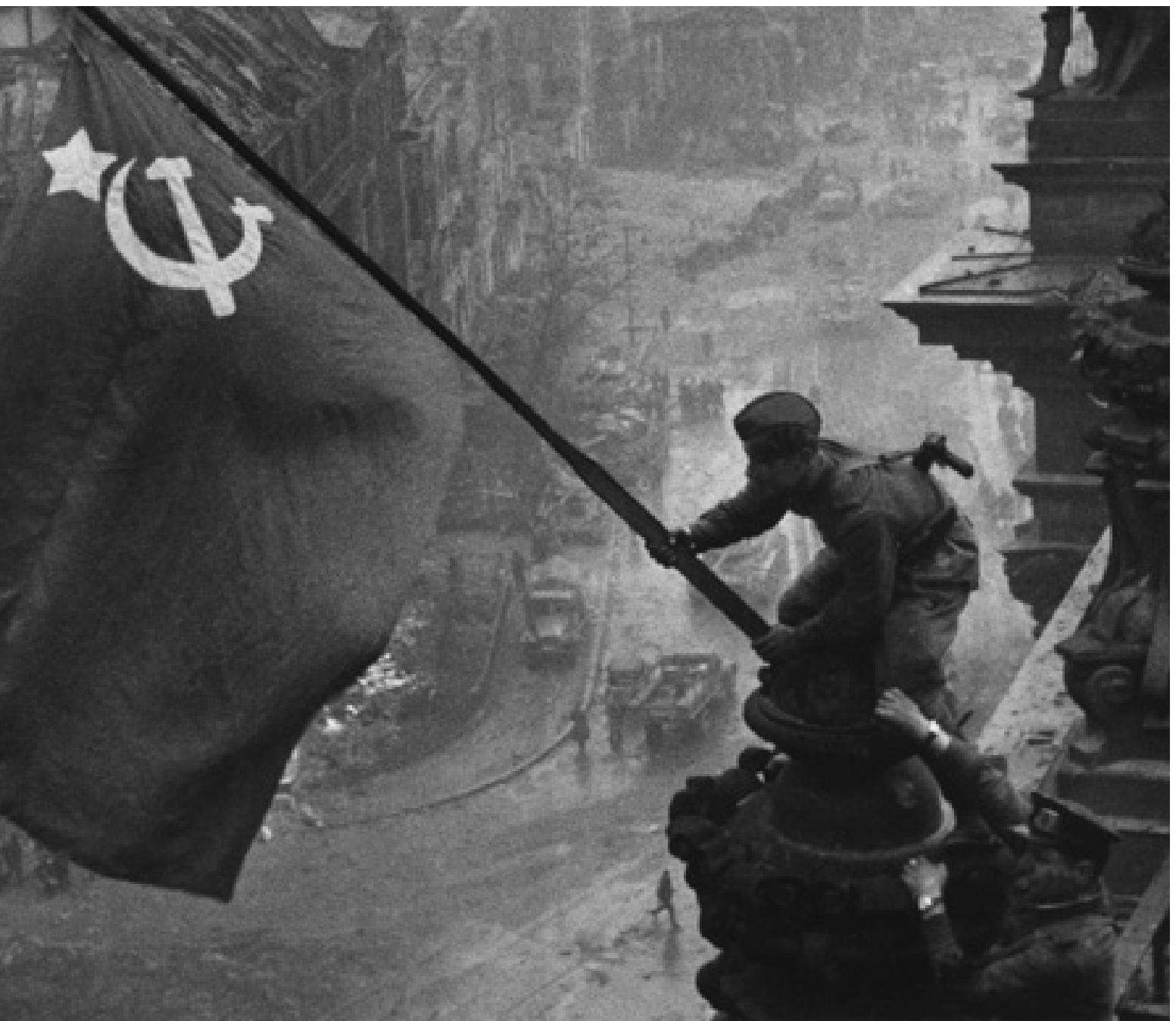}
\includegraphics[width=0.6\textwidth]{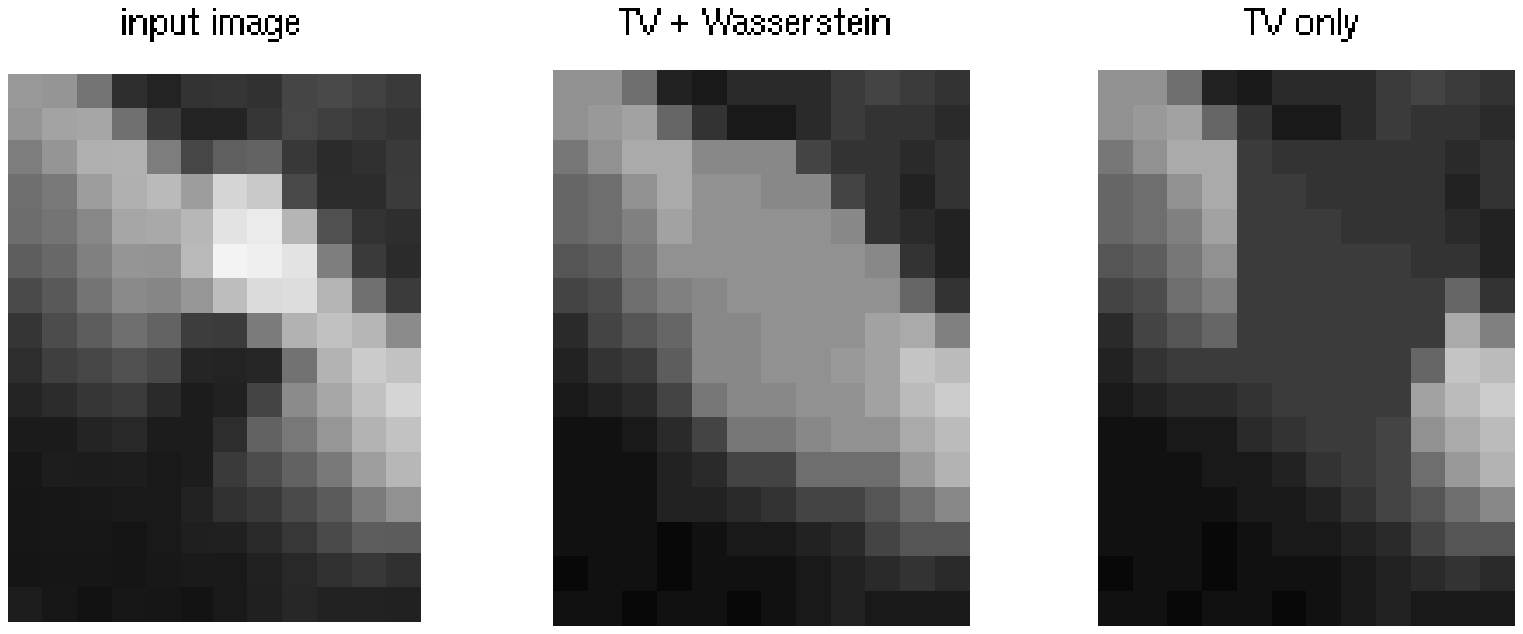}
 \caption{Here we want to inpaint the area occupied by the watch of the
soldier, see the second left image. Our approach, on the second right image
gives better results again than the approach with TV alone.}
 \label{fig:Reichstag_big}
\end{figure}

In the \textbf{sixth} experiment we have a different setup. The original image
is on the left of Figure~\ref{fig:plane}. The histogram
$\mu^0$ was computed from a patch of clouds, which did not include the plane.
The data term is $f(x,y) = \lambda \min(\abs{u_0(x) -
y}^2, \alpha)$, where $\alpha > 0$ is a threshold, so the data term does not
penalize great deviances from the input image too strongly. The Wasserstein term
penalizes the image of the plane whose appearance differs from the prior
statistics. The TV-regularizer is weighted weaker than in the previous examples,
because we do not want to smooth the clouds. 

\emph{Note that unlike in ordinary inpainting applications, we did not specify
the location of the plane beforehand, but the algorithm figured it out on its
own.}
The total variation term finally favors a smooth inpainting of the area occupied
by the plane. In essence we have combined two different tasks: Finding out where
the plane is and inpainting that area occupied by it. See Figure~\ref{fig:plane}
for results.

\begin{figure} 
 \includegraphics[width=\textwidth]{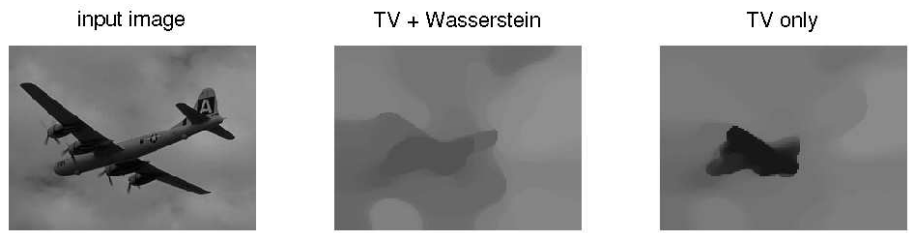}
 \caption{\emph{Unsupervised} inpainting using empirical measures as priors. Objects
not conforming to the prior statistics are removed \emph{without} labeling image
regions.}
 \label{fig:plane}
\end{figure}

\section{Conclusion and Outlook} \label{sec:ConclusionOutlook}
We have presented in this paper a novel method for variational image
regularization, which takes into account global statistical information in one
model. By solving the relaxed nonconvex problem we obtain regularizd images
which conform to some global image statistics, which sets our method apart from
standard variational methods. Moreover, the additional computational cost for the
Wasserstein term we introduced is negligible, however our relaxation is not
tight anymore as in models without the latter term. In our experiments the relaxation was seen to be tight enough for good results.

Our future work will consider extensions of the present approach to
multidimensional input data and related histograms, e.g. based on color, patches
or gradient fields. The theory developed in this paper regarding the possible
exactness of solutions does not carry over without modifications to such more
complex settings. Moreover, it is equally important to find ways related to our
present work to minimize such models efficiently.

\textbf{Acknowledgements.} We would like to thank the anonymous reviewers for their constructive criticism and Marco Esquinazi for helpful discussions.

\bibliographystyle{plain}

\end{document}